\documentclass[preprint,12pt,nonatbib]{elsarticle}

\usepackage{amsmath,amssymb,amsfonts,amsthm,latexsym}
\usepackage{enumerate}
\usepackage[
backend=biber,
style=numeric,
sorting=nyt,
giveninits=false
]{biblatex}
\usepackage{hyperref}
\DeclareNameAlias{author}{family-given}
\DeclareNameAlias{editor}{family-given}
\theoremstyle{plain}
\newtheorem{theorem}{Theorem}[section]

\newtheorem{prop}[theorem]{Proposition}

\theoremstyle{definition}
\newtheorem{defn}[theorem]{Definition}
\newtheorem{example}[theorem]{Example}

\theoremstyle{remark}

\biboptions{sort&compress}

\journal{Linear Algebra and its Applications}

\begin{document}

\begin{frontmatter}

\title{Characterization of Frame-Related Sequences on Semi-Hilbert Spaces}

\author[inst1]{Hemalatha M\corref{cor1}}
\ead{hemapadma28@gmail.com}

\author[inst1]{P. Sam Johnson}
\ead{sam@nitk.edu.in}

\cortext[cor1]{Corresponding author.}

\address[inst1]{Department of Mathematical and Computational Sciences, National Institute of Technology Karnataka, Mangaluru 575025, India}

\begin{abstract}
Let $A\in\mathcal{B}(H)^+$ and $B\in\mathcal{B}(\ell^2)^+$ be positive bounded operators. We investigate reduced weighted adjoints of densely defined closable operators and use them to develop frame-type systems in the semi-Hilbert spaces induced by $A$ and $B$. Closedness, boundedness, range behavior, and algebraic properties of the $A$-adjoint are established, with particular attention to sums, products, and double adjoints in the unbounded setting. We then study the associated $AB$-analysis, $AB$-synthesis, and $AB$-frame operators. Operator-theoretic characterizations of $AB$-Bessel sequences, $AB$-lower semi-frames, and $AB$-frames are obtained, and examples show that several natural weighted-adjoint identities may hold only as proper inclusions. A Douglas-type range characterization for $AB$-frames is also derived.
\end{abstract}


\end{frontmatter}

\section{Introduction}

Let $H$ be a complex Hilbert space and let $A\in\mathcal{B}(H)^+$ be positive. The operator $A$ determines the semi-inner product and seminorm
\[
\langle f,g\rangle_A=\langle Af,g\rangle,
\qquad
\|f\|_A=\langle Af,f\rangle^{1/2},
\qquad f,g\in H.
\]
The pair $(H,\langle\cdot,\cdot\rangle_A)$ is referred to as a semi-Hilbert space. When $A$ is noninjective, the induced form is degenerate; when $\mathcal{R}(A)$ is not closed, completeness and range arguments require additional care. Consequently, familiar Hilbert-space identities cannot always be transferred verbatim to this setting.

Positive semidefinite forms have a long history, beginning with work of Krein \cite{krein1998}. Related weighted operator structures appear in the work of Lax \cite{lax1954} and Dieudonn\'e. A systematic operator theory for semi-Hilbert spaces, including weighted adjoints, projections, and partial isometries, was subsequently developed by Arias, Corach, and Gonz\'alez \cite{antoine1981partial,arias2008partial,arias2008metric}. Closed and unbounded operators in this framework were considered more recently by Baklouti and Namouri \cite{baklouti2022closed}.

For a densely defined operator $T$, the $A$-adjoint $T^{\#}$ is characterized by
\[
AT^{\#}=T^*A,
\qquad
\mathcal{R}(T^{\#})\subseteq\overline{\mathcal{R}(A)}.
\]
Its existence is governed by the range condition $\mathcal{R}(T^*A)\subseteq\mathcal{R}(A)$. In contrast with the ordinary Hilbert-space adjoint, a reduced weighted adjoint need not exist. Moreover, for unbounded operators, formulas for sums, products, and iterated adjoints generally require explicit domain assumptions and may yield inclusions rather than equalities.

Frame theory, initiated by Duffin and Schaeffer \cite{7}, provides a natural source of such operators. Analysis and synthesis operators associated with Bessel sequences, lower semi-frames, and related systems may be unbounded, so weighted adjoint techniques are particularly relevant. We therefore consider two positive weights: $A\in\mathcal{B}(H)^+$ on the signal space and $B\in\mathcal{B}(\ell^2)^+$ on the coefficient space.

The first part of the paper develops basic properties of $A$-adjoints for densely defined closable operators. We examine boundedness, closedness, double-adjoint relations, sums, products, null spaces, and weighted orthogonality. The second part introduces $AB$-Bessel sequences, $AB$-lower semi-frames, and $AB$-frames. Their analysis, synthesis, and frame operators are related to the corresponding reduced weighted adjoints. We obtain range and operator characterizations, exhibit examples in which natural inclusions are strict, and conclude with a Douglas-type criterion for $AB$-frames.

Throughout, $A\in\mathbb{B}(H)^+$ and $B\in\mathbb{B}(\ell^2)^+$ are fixed positive bounded operators. $\mathcal{D}(T)$, $\mathcal{R}(T)$, and $\mathcal{N}(T)$ denote the domain, range, and null space of an operator $T$, respectively. The Hilbert-space adjoint and closure are denoted by $T^*$ and $\overline{T}$. The symbol $T^{\#}$ denotes the relevant reduced weighted adjoint. We write $P_A=P_{\overline{\mathcal{R}(A)}}$ and $\mathbb{N}=\{1,2,\ldots\}$.

\section{Preliminaries}
\begin{prop}\cite{Kreyszig}
Let $T$ and $U$ be densely defined linear operators. Then
\begin{enumerate}
\item If $\mathcal{D}(T+U)$ is dense, then $(T+U)^* \supseteq T^*+U^*.$ Moreover, if $U$ is bounded, then $(T+U)^* = T^*+U^*.$
\item If $\mathcal{D}(TU)$ is dense, then $(TU)^*\supseteq U^*T^*$. Moreover, if $T$ is bounded, then $(TU)^*=U^*T^*$.
\item If $T\subseteq U$, then $U^*\subseteq T^*$.
\end{enumerate}
\end{prop}

\begin{defn}\cite{kato}
\begin{enumerate}
\item A linear operator $A$ on $H$ is called \emph{closed} if its graph
$G(A)=\{(f,Af):f\in \mathcal{D}(A)\}$
is a closed subspace of $H\times H$.

\item Let $A$ be a densely defined closed operator on $H$. The
\emph{reduced minimum modulus} of $A$ is
\[
\gamma(A)
=\inf\left\{\frac{\|Af\|}{\|f\|}:
0\neq f\in\mathcal{D}(A)\cap\mathcal{N}(A)^\perp\right\}.
\]
\end{enumerate}
\end{defn}

\begin{theorem}\cite{limaye}
Let $H$ and $K$ be Banach spaces and let $F:H\rightarrow K$ be a closed linear operator with $\mathcal{D}(F)=H$. Then $F$ is bounded.
\end{theorem}

\begin{theorem}\cite{kato}\label{closed}
Let $A$ be a densely defined closed operator on $H$. Then the following conditions are equivalent.
\begin{enumerate}
\item $\gamma(A)>0$.
\item $\mathcal{R}(A)$ is closed.
\item $\mathcal{R}(A^*)$ is closed.
\end{enumerate}
\end{theorem}

\begin{theorem}\cite{von}\label{von}
Let $H$ be a separable Hilbert space and let $A$ be an unbounded self-adjoint operator on $H$. Then there exists a unitary operator $U$ such that
$\mathcal{D}(U^*AU)\cap \mathcal{D}(A)=\{0\}$.
\end{theorem}
\begin{theorem}\cite{douglas1966majorization}\label{douglas}
If $S$ and $T$ are closed densely defined operators on a Hilbert space $H$, then the following statements hold.
\begin{enumerate}
\item[(i)] If $SS^{*}\leq TT^{*}$, then there exists a contraction $V$ such that $S\subseteq TV$.

\item[(ii)] If $\mathcal{R}(S)\subseteq\mathcal{R}(T)$, then there exists a unique densely defined operator $V$ satisfying $S=TV$ and $\mathcal{R}(V)\subseteq\mathcal{R}(T^\dagger)$. Moreover, there exists a constant $M\geq0$ such that $\|Vf\|^{2}\leq M\left(\|f\|^{2}+\|Sf\|^{2}\right)$ for every $f\in \mathcal{D}(V)$. Furthermore, $V$ is bounded whenever $S$ is bounded and $V$ is closed whenever $T$ is bounded.
\end{enumerate}
\end{theorem}
\begin{theorem}\cite{forough2014majorization}\label{banach}
Suppose $T$ and $S$ are closed densely defined operators on a Banach space $X$ and assume that $T$ has a generalized inverse $T^\dagger$. If $\mathcal{R}(S)\subseteq \mathcal{R}(T)$, then there exists a unique densely defined operator $V$ on $X$ such that $S=TV$ and $\mathcal{R}(V)\subseteq \mathcal{R}(T^\dagger)$. The operator $V$ is called the \emph{reduced solution} of the equation $S=TX$. Moreover, there exists a constant $M\geq0$ such that
\[
\|Vf\|^{2}\leq M\bigl(\|f\|^{2}+\|Sf\|^{2}\bigr), \text{ for } f\in \mathcal{D}(V).
\]
Furthermore, if $S$ is bounded, then $V$ is bounded and if $T$ is bounded, then $V$ is closed.
\end{theorem}
\begin{defn}\cite{1,2002,7}
Let $F=\{f_n\}_{n\in\mathbb{N}}$ be a sequence in a separable Hilbert space $H$.
\begin{enumerate}[(i)]
\item $F$ is \emph{complete} if
\[
\overline{\operatorname{span}\{f_n:n\in\mathbb{N}\}}=H.
\]
\item $F$ is a \emph{frame} for $H$ if there are constants $0<a\leq b<\infty$ such that
\[
a\|f\|^2\leq\sum_{n=1}^{\infty}|\langle f,f_n\rangle|^2\leq b\|f\|^2,
\qquad f\in H.
\]
\item $F$ is a \emph{Bessel sequence} if the upper estimate above holds for some $b>0$.
\item $F$ is an \emph{upper semi-frame} if it is complete and Bessel; equivalently,
\[
0<\sum_{n=1}^{\infty}|\langle f,f_n\rangle|^2\leq b\|f\|^2,
\qquad 0\neq f\in H,
\]
for some $b>0$.
\item $F$ is a \emph{lower semi-frame} if there is $a>0$ such that
\[
a\|f\|^2\leq\sum_{n=1}^{\infty}|\langle f,f_n\rangle|^2,
\qquad f\in H.
\]
\end{enumerate}
\end{defn}

For any given sequence $F = \{f_n\}_{n \in \mathbb{N}} \subseteq H$, we can associate the following three operators:
\begin{enumerate}[(i)]
\item The analysis operator $C_F: \mathcal{D}(C_F) \subseteq H \rightarrow \ell^2(\mathbb{N})$ is defined by \begin{eqnarray*}
C_Ff &=& \{\langle f, f_n \rangle\}, \hspace*{0.3cm} \text{ for all } f \in \mathcal{D}(C_F) \\\mathcal{D}(C_F) &=& \{f \in H : \sum_{n=1}^{\infty} | \langle f, f_n \rangle |^2 < \infty \}.
\end{eqnarray*}
\item The synthesis operator $D_F: \mathcal{D}(D_F) \subseteq \ell^2(\mathbb{N}) \rightarrow H$ is defined by
\begin{eqnarray*}
D_F(\{c_n\}_{n \in \mathbb{N}}) &=& \sum_{n=1}^{\infty} c_nf_n, \hspace*{0.3cm} \text{ for all } \{c_n\} \in \mathcal{D}(D_F)\\
\mathcal{D}(D_F) &=& \Big\{\{c_n\}_{n \in \mathbb{N}} \in \ell^2(\mathbb{N}) : \sum_{n=1}^{\infty} c_nf_n \text{ is convergent in }H \Big\}.
\end{eqnarray*}

\item The frame operator $S_F: \mathcal{D}(S_F) \subseteq H \rightarrow H $ is defined by
\begin{eqnarray*}
S_Ff &=& \sum_{n=1}^{\infty} \langle f, f_n \rangle f_n, \text{ for all } f \in \mathcal{D}(S_F)\\
\mathcal{D}(S_F) &=& \Big\{f \in H : \sum_{n=1}^{\infty} \langle f, f_n \rangle f_n \text{ is convergent in }H \Big\}.
\end{eqnarray*}
\end{enumerate}
\section{The $A$-Adjoint Operator}

\begin{defn}
Let $A\in\mathcal{B}(H)^+$. A linear operator $T:\mathcal{D}(T)\subseteq H\to H$ is called \emph{$A$-bounded} if there is a constant $c\geq0$ such that
\[
\|Tf\|_A\leq c\|f\|_A,
\qquad f\in\mathcal{D}(T).
\]
The least admissible constant is denoted by $\|T\|_A$.
\end{defn}

\begin{defn}
Let $A\in\mathcal{B}(H)^+$ and let $T:\mathcal{D}(T)\subseteq H\to H$ be linear. The operator $T$ is \emph{$A$-closed} if, whenever $\{f_k\}_{k\in\mathbb{N}}\subseteq\mathcal{D}(T)\cap\overline{\mathcal{R}(A)}$ satisfies
\[
\|f_k-f\|_A\to0,
\qquad
\|Tf_k-g\|_A\to0
\]
for some $f,g\in H$, one has $f\in\mathcal{D}(T)$ and $Tf-g\in\mathcal{N}(A)$.
\end{defn}

\begin{defn}
Let $A\in\mathcal{B}(H)^+$, and let $T:\mathcal{D}(T)\subseteq H\to H$ be densely defined and closable. Assume that
\[
\mathcal{R}(T^*A)\subseteq\mathcal{R}(A).
\]
Then by Theorems \ref{douglas} and \ref{banach}, we can conclude that there exists an unique closed operator $T^\#$ such that $AT^\# = T^*A.$ Then, the \emph{$A$-adjoint} of $T$ is the operator $T^{\#}$ defined by
\[
\mathcal{D}(T^{\#})
=
\left\{
f\in H:
\begin{array}{l}
\text{there is a unique }g_f\in\overline{\mathcal{R}(A)}\text{ such that}\\[-1mm]
\langle Th,f\rangle_A=\langle h,g_f\rangle_A
\text{ for every }h\in\mathcal{D}(T)
\end{array}
\right\},
\]
and $T^{\#}f=g_f$ for $f\in\mathcal{D}(T^{\#})$.
\end{defn}

Set
\[
\mathcal{L}_A(H)=\{T\in\mathcal{L}(H):T^{\#}\text{ exists}\}.
\]
In general, $\mathcal{L}_A(H)$ is not a linear space. The next example shows that two operators may each admit a $A$-adjoint while their sum does not.

\begin{example}
The class $\mathcal{L}_A(H)$ need not be closed under addition. Let $H=\ell^2(\mathbb{N})$ and
\[
A=\operatorname{diag}\left(1,\frac12,\frac13,\ldots\right).
\]
Then $A$ is positive and injective, and $\mathcal{R}(A)$ is dense but not closed. Its inverse
\[
A^{-1}=\operatorname{diag}(1,2,3,\ldots),
\qquad
\mathcal{D}(A^{-1})=\mathcal{R}(A),
\]
is an unbounded self-adjoint operator. By Theorem~\ref{von}, there is a unitary operator $U$ such that
\[
\mathcal{D}(U^*A^{-1}U)\cap\mathcal{D}(A^{-1})=\{0\}.
\]
Set $X=U^*A^{-1}U$. Then $X$ is self-adjoint, closed, and densely defined, and
\[
\mathcal{D}(X)\cap\mathcal{R}(A)=\{0\}.
\]
Let $T=X^*$, so that $T$ is closed and densely defined and $T^*=X$.

Choose $f_0\in H\setminus\mathcal{R}(A)$ and $g_0\in H$ with $Ag_0\neq0$. Define the bounded rank-one operator
\[
Zf=\langle f,g_0\rangle f_0,
\qquad f\in H,
\]
and put $Y=Z-X$ and $S=Y^*$. Since $Z$ is bounded, $Y$ is closed and densely defined with $\mathcal{D}(Y)=\mathcal{D}(X)$, while $S=Z^*-T$.

Because $\mathcal{D}(X)\cap\mathcal{R}(A)=\{0\}$,
\[
\mathcal{D}(T^*A)
=\{f\in H:Af\in\mathcal{D}(X)\}
=\{0\},
\]
and hence $\mathcal{R}(T^*A)=\{0\}\subseteq\mathcal{R}(A)$. The same argument, using $\mathcal{D}(Y)=\mathcal{D}(X)$, yields
\[
\mathcal{D}(S^*A)=\{0\},
\qquad
\mathcal{R}(S^*A)=\{0\}\subseteq\mathcal{R}(A).
\]
Thus both $T$ and $S$ admit $A$-adjoints.

On $\mathcal{D}(T)\cap\mathcal{D}(S)$, one has $T+S=Z^*$. Since $Z^*$ is bounded, $(T+S)^*=Z$. Consequently,
\begin{align*}
(T+S)^*A(Ag_0)
&=ZA^2g_0\\
&=\langle A^2g_0,g_0\rangle f_0\\
&=\|Ag_0\|^2f_0.
\end{align*}
The last vector does not belong to $\mathcal{R}(A)$ because $Ag_0\neq0$ and $f_0\notin\mathcal{R}(A)$. Therefore,
\[
\mathcal{R}((T+S)^*A)\nsubseteq\mathcal{R}(A),
\]
so $T+S$ has no $A$-adjoint.
\end{example}

\begin{prop}[Cauchy--Schwarz inequality]
Let $A\in\mathcal{B}(H)^+$. Then
\[
|\langle f,g\rangle_A|\leq\|f\|_A\|g\|_A,
\qquad f,g\in H.
\]
\end{prop}

\begin{proof}
If $g\in\mathcal{N}(A)$, then $\langle f,g\rangle_A=0$ for every $f\in H$. Assume that $\|g\|_A>0$. Positivity gives
\[
0\leq\|f-\lambda g\|_A^2
=\|f\|_A^2-\overline{\lambda}\langle f,g\rangle_A
-\lambda\langle g,f\rangle_A+|\lambda|^2\|g\|_A^2.
\]
Taking $\lambda=\langle f,g\rangle_A/\|g\|_A^2$ yields
\[
0\leq\|f\|_A^2-
\frac{|\langle f,g\rangle_A|^2}{\|g\|_A^2},
\]
which is the desired estimate.
\end{proof}

\begin{theorem}[Riesz representation in the semi-Hilbert setting]
Let $\phi:H\to\mathbb{C}$ be linear. If there is $C>0$ such that
\[
|\phi(f)|\leq C\|f\|_A,
\qquad f\in H,
\]
then there is a unique $g\in\overline{\mathcal{R}(A)}$ satisfying
\[
\phi(f)=\langle f,g\rangle_A,
\qquad f\in H.
\]
\end{theorem}

\begin{proof}
The assumed estimate implies $\mathcal{N}(A)\subseteq\mathcal{N}(\phi)$. Hence
\[
\widetilde\phi\bigl(f+\mathcal{N}(A)\bigr)=\phi(f)
\]
defines a linear functional on $H/\mathcal{N}(A)$. It is well defined because two representatives of the same coset differ by an element of $\mathcal{N}(A)$. Moreover,
\[
\left|\widetilde\phi\bigl(f+\mathcal{N}(A)\bigr)\right|
\leq C\|f\|_A,
\]
so $\widetilde\phi$ extends continuously to the completion $H_A$ of $H/\mathcal{N}(A)$.

By the classical Riesz representation theorem, there is a unique $u\in H_A$ such that
\[
\widetilde\phi(\xi)=\langle\xi,u\rangle_{H_A},
\qquad \xi\in H_A.
\]
Under the canonical realization of $H_A$ in $\overline{\mathcal{R}(A)}$, let $g\in\overline{\mathcal{R}(A)}$ correspond to $u$. Then
\[
\phi(f)=\widetilde\phi\bigl(f+\mathcal{N}(A)\bigr)
=\langle f,g\rangle_A,
\qquad f\in H.
\]

For uniqueness, suppose that $g_1,g_2\in\overline{\mathcal{R}(A)}$ both represent $\phi$. Then
\[
\langle f,g_1-g_2\rangle_A=0,
\qquad f\in H.
\]
Taking $f=g_1-g_2$ gives $g_1-g_2\in\mathcal{N}(A)$. Since
\[
\mathcal{N}(A)\cap\overline{\mathcal{R}(A)}=\{0\},
\]
one obtains $g_1=g_2$.
\end{proof}

\begin{prop}
Let $T:H\to H$ be an $A$-bounded operator. Then the following statements hold.
\begin{enumerate}
\item $T$ admits a unique $A$-adjoint $T^{\#}$.
\item The operator $T^{\#}$ is $A$-bounded. Moreover, $\|T\|_{A}=\|T^{\#}\|_{A}$.
\end{enumerate}
\end{prop}
\begin{proof}
(1) Let $M\geq0$ satisfy
\[
\|Tf\|_A\leq M\|f\|_A,
\qquad f\in H.
\]
For fixed $g\in H$, define
$\varphi_g(f)=\langle Tf,g\rangle_A$. The Cauchy--Schwarz inequality gives
\[
|\varphi_g(f)|
\leq M\|f\|_A\|g\|_A.
\]
By the semi-Hilbert-space Riesz representation theorem, there is a unique
$f_g\in\overline{\mathcal{R}(A)}$ such that
\[
\langle Tf,g\rangle_A=\langle f,f_g\rangle_A,
\qquad f\in H.
\]
Define $T^{\#}g=f_g$. This gives the unique $A$-adjoint of $T$.

(2) The construction yields
\[
\|T^{\#}g\|_A\leq M\|g\|_A,
\qquad g\in H,
\]
so $\|T^{\#}\|_A\leq\|T\|_A$. Conversely, for $f\in H$,
\begin{align*}
\|Tf\|_A
&=\sup_{\|g\|_A\leq1}|\langle Tf,g\rangle_A|\\
&=\sup_{\|g\|_A\leq1}|\langle f,T^{\#}g\rangle_A|\\
&\leq\|T^{\#}\|_A\|f\|_A.
\end{align*}
Therefore $\|T\|_A\leq\|T^{\#}\|_A$, and equality follows.
\end{proof}
\begin{example}
A densely defined $A$-bounded operator need not admit a $A$-adjoint. Let
$H=\ell^2$ and
\[
A=\operatorname{diag}(1,0,0,\ldots).
\]
Thus, $\|f\|_A=|\xi_1|$ for $f=(\xi_n)_{n\in\mathbb{N}}\in\ell^2$. Set
$\mathcal{D}(T)=c_{00}$ and define
\[
Tf=\left(0,\sum_{n=1}^{\infty}n\xi_n,0,0,\ldots\right),
\qquad f=(\xi_n)_{n\in\mathbb{N}}\in c_{00}.
\]
The sum is finite on $c_{00}$, so $T$ is well defined and densely defined. Since the first coordinate of $Tf$ is zero,
\[
\|Tf\|_A=0\leq\|f\|_A,
\qquad f\in\mathcal{D}(T),
\]
and hence $T$ is $A$-bounded.

For $k\in\mathbb{N}$, let $f^{(k)}=k^{-1}e_k$. Then $f^{(k)}\to0$ in $\ell^2$, whereas
\[
Tf^{(k)}=e_2
\qquad\text{for every }k\in\mathbb{N}.
\]
Therefore, $T$ is not closable. In particular, $A$-boundedness on a dense domain does not by itself guarantee the existence of a $A$-adjoint.
\end{example}
\begin{theorem}\label{bound}
Let $T:\mathcal{D}(T)\subseteq H\to H$ be a densely defined closable operator. Suppose that the $A$-adjoint $T^{\#}$ exists and is densely defined. Then $T$ is $A$-bounded if and only if $T^{\#}$ is $A$-bounded. Moreover, their $A$-bounds are equal, that is, $\|T\|_A=\|T^{\#}\|_A$.
\end{theorem}

\begin{proof}
Since
\[
\|f\|_A\leq\|A\|^{1/2}\|f\|,
\qquad f\in H,
\]
Hilbert-space density of $\mathcal{D}(T)$ and $\mathcal{D}(T^{\#})$ implies
density with respect to the $A$-seminorm.

Assume first that $T$ is $A$-bounded. For $g\in\mathcal{D}(T^{\#})$,
\begin{align*}
\|T^{\#}g\|_A
&=\sup\bigl\{|\langle f,T^{\#}g\rangle_A|:
f\in\mathcal{D}(T),\ \|f\|_A\leq1\bigr\}\\
&=\sup\bigl\{|\langle Tf,g\rangle_A|:
f\in\mathcal{D}(T),\ \|f\|_A\leq1\bigr\}\\
&\leq \|T\|_A\|g\|_A.
\end{align*}
Thus $T^{\#}$ is $A$-bounded and
$\|T^{\#}\|_A\leq\|T\|_A$.

Conversely, suppose that $T^{\#}$ is $A$-bounded. For
$f\in\mathcal{D}(T)$,
\begin{align*}
\|Tf\|_A
&=\sup\bigl\{|\langle Tf,g\rangle_A|:
g\in\mathcal{D}(T^{\#}),\ \|g\|_A\leq1\bigr\}\\
&=\sup\bigl\{|\langle f,T^{\#}g\rangle_A|:
g\in\mathcal{D}(T^{\#}),\ \|g\|_A\leq1\bigr\}\\
&\leq \|T^{\#}\|_A\|f\|_A.
\end{align*}
Hence $T$ is $A$-bounded and
$\|T\|_A\leq\|T^{\#}\|_A$. The two inequalities yield the asserted equality.
\end{proof}

\begin{theorem}
Let $T:\mathcal{D}(T)\subseteq H\to H$ be a densely defined closable operator satisfying $\mathcal{R}(T^*A)\subseteq\mathcal{R}(A)$. Then the following statements hold.
\begin{enumerate}
\item $T^\#$ is a closed operator.
\item $T^\#$ is an $A$-closed operator.
\item If $(T^\#)^\#$ exists, then $P_A\overline T\subseteq(T^\#)^\#$. Equivalently, $A\overline T\subseteq A(T^\#)^\#=(T^\#)^*A$.
\end{enumerate}
If, in addition, $(T^*A)^*=A\overline T$, then $(T^\#)^\#$ exists and $(T^\#)^\#=P_A\overline T$. In particular, $A(T^\#)^\#=A\overline T$.
\end{theorem}
\begin{proof}
(1) Since $T^\#$ is the $A$-adjoint of $T$, we have $T^\#=A^\dagger T^*A$ and $\mathcal{R}(T^\#)\subseteq\overline{\mathcal{R}(A)}$. Let $\{f_n\}\subseteq\mathcal{D}(T^\#)$ satisfy $f_n\to f$ and $T^\#f_n\to g$ in $H$. Since $A$ is bounded, $Af_n\to Af$ and $AT^\#f_n\to Ag$. Moreover, $AT^\#=T^*A$, and therefore $T^*Af_n\to Ag$. Since $T^*$ is closed, it follows that $Af\in\mathcal{D}(T^*)$ and $T^*Af=Ag$. In particular, $T^*Af\in\mathcal{R}(A)$, so $f\in\mathcal{D}(T^\#)$. Hence $T^\#f=A^\dagger T^*Af=A^\dagger Ag=P_Ag$. Since $T^\#f_n\in\overline{\mathcal{R}(A)}$ for every $n$ and $\overline{\mathcal{R}(A)}$ is closed, we have $g\in\overline{\mathcal{R}(A)}$. Thus, $P_Ag=g$, and consequently $T^\#f=g$. Therefore, $T^\#$ is closed.

(2) Let $\{f_n\}\subseteq\mathcal{D}(T^\#)\cap\overline{\mathcal{R}(A)}$ satisfy $\|f_n-f\|_A\to0$ and $\|T^\#f_n-g\|_A\to0$ for some $f,g\in H$. Since $\|Ah\|^2\leq\|A\|\|h\|_A^2$ for every $h\in H$, it follows that $Af_n\to Af$ and $AT^\#f_n\to Ag$ in the Hilbert space norm. Using $AT^\#=T^*A$, we obtain $T^*Af_n\to Ag$. Since $T^*$ is closed, we have $Af\in\mathcal{D}(T^*)$ and $T^*Af=Ag$. Hence $f\in\mathcal{D}(T^\#)$ and $AT^\#f=T^*Af=Ag$. Therefore, $T^\#f-g\in\mathcal{N}(A)$, proving that $T^\#$ is $A$-closed.

(3) Suppose that $(T^\#)^\#$ exists. Since $AT^\#=T^*A$, taking Hilbert space adjoint gives $(T^\#)^*A=(AT^\#)^*=(T^*A)^*$. Since $A$ is bounded and $T$ is closable, the standard adjoint relation gives $A\overline T=AT^{**}\subseteq(T^*A)^*$. Therefore, $A\overline T\subseteq(T^\#)^*A=A(T^\#)^\#$.

Let $f\in\mathcal{D}(\overline T)$. Then $f\in\mathcal{D}((T^\#)^\#)$ and $A(T^\#)^\#f=A\overline Tf$. Since $AP_A=A$, we also have $AP_A\overline Tf=A\overline Tf$. Hence $(T^\#)^\#f-P_A\overline Tf\in\mathcal{N}(A)$. Both $(T^\#)^\#f$ and $P_A\overline Tf$ belong to $\overline{\mathcal{R}(A)}$. Since $\mathcal{N}(A)\cap\overline{\mathcal{R}(A)}=\{0\}$, it follows that $(T^\#)^\#f=P_A\overline Tf$. Thus, $P_A\overline T\subseteq(T^\#)^\#$.

Finally, suppose that $(T^*A)^*=A\overline T$. Then $(T^\#)^*A=(T^*A)^*=A\overline T$, and therefore $\mathcal{R}((T^\#)^*A)\subseteq\mathcal{R}(A)$. Hence $(T^\#)^\#$ exists. Moreover, $A(T^\#)^\#=(T^\#)^*A=A\overline T$. Since both $(T^\#)^\#$ and $P_A\overline T$ have range contained in $\overline{\mathcal{R}(A)}$, the uniqueness of the reduced solution gives $(T^\#)^\#=P_A\overline T$.
\end{proof}
\begin{theorem}
Let $T$ and $S$ be densely defined closable operators on $H$. Suppose that
\[
\mathcal{R}(T^*A)\subseteq\mathcal{R}(A),
\qquad
\mathcal{R}(S^*A)\subseteq\mathcal{R}(A).
\]
Then the following statements hold.
\begin{enumerate}
\item If $(T+S)^\#$ exists, then
$T^\#+S^\#\subseteq(T+S)^\#$.
\item If $S$ is bounded, then
$(T+S)^\#=T^\#+S^\#$.
\end{enumerate}
\end{theorem}

\begin{proof}
(1) Let $f\in\mathcal{D}(S^\#T^\#)$. Then
$f\in\mathcal{D}(T^\#)$ and $T^\#f\in\mathcal{D}(S^\#)$. From
$AT^\#=T^*A$ and $AS^\#=S^*A$, we obtain
\[
AS^\#T^\#f=S^*T^*Af.
\]
Since $S^*T^*\subseteq(TS)^*$ and $(TS)^\#$ exists,
\[
AS^\#T^\#f=(TS)^*Af=A(TS)^\#f.
\]
Thus $\bigl(S^\#T^\#-(TS)^\#\bigr)f$ belongs to both
$\mathcal{N}(A)$ and $\overline{\mathcal{R}(A)}$, and hence it is zero.
Therefore,
\[
S^\#T^\#\subseteq(TS)^\#.
\]

(2) If $S\in\mathcal{B}(H)$, then $(TS)^*=S^*T^*$, and
\begin{align*}
(TS)^\#
&=A^\dagger S^*T^*A\\
&=A^\dagger S^*AA^\dagger T^*A\\
&=(A^\dagger S^*A)(A^\dagger T^*A)\\
&=S^\#T^\#.
\end{align*}
Here $AA^\dagger T^*A=T^*A$ because
$\mathcal{R}(T^*A)\subseteq\mathcal{R}(A)$.
\end{proof}
\begin{theorem}
Let $T$ and $S$ be densely defined closable operators on $H$ and suppose that they admit $A$-adjoints
$T^{\#}$ and $S^{\#}$, respectively. Then
\begin{enumerate}
\item If $(TS)^\#$ exists, then $S^\#T^\#\subseteq(TS)^\#$.
\item If $S \in \mathcal{B}(H)$ and $(TS)^\#$ exists, then
$(TS)^\#=S^\#T^\#$.
\end{enumerate}
\end{theorem}

\begin{proof}
(1) Let $f\in\mathcal{D}(S^\#T^\#)$. Since
\[
AT^\#=T^*A,\qquad AS^\#=S^*A,
\]
we obtain
\[
AS^\#T^\#f=S^*AT^\#f=S^*T^*Af.
\]
Because $T$ and $S$ are densely defined and closable,
$S^*T^*\subseteq(TS)^*$. Hence
\[
AS^\#T^\#f=(TS)^*Af=A(TS)^\#f.
\]
It follows that
\[
\bigl(S^\#T^\#-(TS)^\#\bigr)f
\in \mathcal{N}(A)\cap\overline{\mathcal{R}(A)}=\{0\}.
\]
Thus $S^\#T^\#f=(TS)^\#f$, and therefore
$S^\#T^\#\subseteq(TS)^\#$.

(2) If $S\in\mathcal{B}(H)$, then $(TS)^*=S^*T^*$. Consequently,
\begin{align*}
(TS)^\#
&=A^\dagger (TS)^*A \\
&=A^\dagger S^*T^*A \\
&=A^\dagger S^*AA^\dagger T^*A \\
&=(A^\dagger S^*A)(A^\dagger T^*A) \\
&=S^\#T^\#.
\end{align*}
Here $AA^\dagger T^*A=T^*A$ because
$\mathcal{R}(T^*A)\subseteq\mathcal{R}(A)$.
\end{proof}
\begin{theorem}
Let $F:\mathcal{D}(F)\subseteq H\to H$ be a densely defined operator admitting a $A$-adjoint $F^\#$. Then $\mathcal{R}(F)^{\perp_A}=\mathcal{N}(F^\#)$, where $\mathcal{R}(F)^{\perp_A}=\{f\in H:\langle f,h\rangle_A=0\text{ for every }h\in\mathcal{R}(F)\}$.
\end{theorem}

\begin{proof}
Let $f\in\mathcal{R}(F)^{\perp_A}$. Then $\langle f,Fg\rangle_A=0$ for every $g\in\mathcal{D}(F)$. Since $\langle f,Fg\rangle_A=\langle Af,Fg\rangle$, it follows that $\langle Fg,Af\rangle=0$ for every $g\in\mathcal{D}(F)$. Thus, the linear functional $g\mapsto\langle Fg,Af\rangle$ is bounded on $\mathcal{D}(F)$. By the definition of the Hilbert space adjoint, $Af\in\mathcal{D}(F^*)$ and $F^*Af=0$. In particular, $F^*Af\in\mathcal{R}(A)$, and hence $f\in\mathcal{D}(F^\#)$. Since $F^\#$ is the $A$-adjoint of $F$, we have $AF^\#f=F^*Af=0$. Therefore, $F^\#f\in\mathcal{N}(A)$. On the other hand, $\mathcal{R}(F^\#)\subseteq\overline{\mathcal{R}(A)}$, and hence $F^\#f\in\mathcal{N}(A)\cap\overline{\mathcal{R}(A)}=\{0\}$. Thus, $F^\#f=0$, proving that $\mathcal{R}(F)^{\perp_A}\subseteq\mathcal{N}(F^\#)$.

Conversely, let $f\in\mathcal{N}(F^\#)$. Then $f\in\mathcal{D}(F^\#)$ and $F^\#f=0$. Since $AF^\#=F^*A$ on $\mathcal{D}(F^\#)$, we have $F^*Af=AF^\#f=0$. Therefore, for every $g\in\mathcal{D}(F)$, $\langle f,Fg\rangle_A=\langle Af,Fg\rangle=\langle F^*Af,g\rangle=0$. Hence $f\in\mathcal{R}(F)^{\perp_A}$, and consequently $\mathcal{N}(F^\#)\subseteq\mathcal{R}(F)^{\perp_A}$. Therefore, $\mathcal{R}(F)^{\perp_A}=\mathcal{N}(F^\#)$.
\end{proof}
\section{$AB$-Frame-Related Sequences}

Let $F=\{f_n\}_{n\in\mathbb{N}}\subseteq H$. We use $\|c\|_B=\langle Bc,c\rangle^{1/2}, \qquad c\in\ell^2.$ 

The weighted analysis operator is
\[
C_{AB}=\{\langle f,f_n\rangle_A\}_{n\in\mathbb{N}}=C_FA,
\]
with natural domain
\[
\mathcal{D}(C_{AB})
=\left\{f\in H:\{\langle f,f_n\rangle_A\}_{n\in\mathbb{N}}\in\ell^2\right\}.\]
The weighted synthesis operator is
\[
D_{AB}=\sum_{n=1}^{\infty}(Bc)_nf_n=D_FB,
\]
where
\[
\mathcal{D}(D_{AB})
=\{c\in\ell^2:Bc\in\mathcal{D}(D_F)\}.
\]
And the weighted frame operator is,
\[
S_{AB}=D_{AB}C_{AB},
\]
with
\[
\mathcal{D}(S_{AB})
=\{f\in\mathcal{D}(C_{AB}):C_{AB}f\in\mathcal{D}(D_{AB})\}.
\]

\begin{defn}
The sequence $F$ is an \emph{$AB$-Bessel sequence} with bound $\beta>0$ if $\mathcal{D}(C_{AB})=H$ and
\[
\|C_{AB}f\|_B^2\leq\beta\|f\|_A^2,
\text{ for all } f\in H.
\]
\end{defn}

\begin{defn}
The sequence $F$ is an \emph{$AB$-lower semi-frame} with lower bound $\alpha>0$ if
\[
\alpha\|f\|_A^2\leq\|C_{AB}f\|_B^2,
\text{ for all } f\in\mathcal{D}(C_{AB}).
\]
\end{defn}

\begin{defn}
The sequence $F$ is an \emph{$AB$-frame} with bounds $\alpha,\beta>0$ if $\mathcal{D}(C_{AB})=H$ and
\[
\alpha\|f\|_A^2
\leq\|C_{AB}f\|_B^2
\leq\beta\|f\|_A^2,
\text{ for all } f\in H.
\]
\end{defn}

\begin{prop}\label{closedoperator}
Let 
$F=\{f_n\}_{n\in\mathbb{N}}\subset H$. Then the $AB$-analysis operator
\[
C_{AB}:\mathcal{D}(C_{AB})\subseteq H\to\ell^2,
\qquad
C_{AB}f=\{\langle f,f_n\rangle_A\}_{n\in\mathbb{N}},
\]
is closed with respect to the indicated topologies.
\end{prop}

\begin{proof}
Let $\{f^{(k)}\}_{k\in\mathbb{N}}\subseteq\mathcal{D}(C_{AB})$ satisfy
\[
f^{(k)}\to f\quad\text{in }H,
\qquad
C_{AB}f^{(k)}\to d\quad\text{in the coefficient space}.
\]
For each $n\in\mathbb{N}$, boundedness of $A$ gives
\[
\langle f^{(k)},f_n\rangle_A
=\langle Af^{(k)},f_n\rangle
\longrightarrow
\langle Af,f_n\rangle
=\langle f,f_n\rangle_A.
\]
On the other hand, convergence of $C_{AB}f^{(k)}$ yields convergence of the corresponding coordinates to $d_n$. Hence
$d_n=\langle f,f_n\rangle_A$ for every $n$, so
\[
d=\{\langle f,f_n\rangle_A\}_{n\in\mathbb{N}}.
\]
Thus $f\in\mathcal{D}(C_{AB})$ and $C_{AB}f=d$, proving that the graph of $C_{AB}$ is closed.
\end{proof}
\begin{prop}
Let 
$F=\{f_n\}_{n\in\mathbb{N}}\subset H$. If $B(c_{00})\subseteq c_{00}$, where
$c_{00}$ is the space of finitely supported sequences, then $D_{AB}$ is densely defined.
\end{prop}
\begin{proof}
Let $c=\{c_n\}_{n\in\mathbb{N}}\in c_{00}$. Since $B(c_{00})\subseteq c_{00}$, we have $Bc=\{b_n\}_{n\in\mathbb{N}}\in c_{00}$. Hence, there exists $N\in\mathbb{N}$ such that $b_n=0$ for all $n>N$. Therefore,
$D_{AB}c=\sum_{n=1}^{N}b_nf_n$, which is a finite sum and thus converges in $(H,\langle\cdot,\cdot\rangle_A)$. Consequently, $c_{00}\subseteq \mathcal{D}(D_{AB})$. Since $c_{00}$ is dense in $\ell^2$, it follows that $\mathcal{D}(D_{AB})$ is dense in $\ell^2$. Hence, $D_{AB}$ is densely defined.
\end{proof}
\begin{example}
The assumption $B(c_{00})\subseteq c_{00}$ is essential. Let $H=\ell^2$, take
$A=I_H$, and set $f_n=ne_n$, where $\{e_n\}_{n\in\mathbb{N}}$ is the canonical orthonormal basis. Then
\[
D_Fc=(nc_n)_{n\in\mathbb{N}},
\qquad
\mathcal{D}(D_F)
=\left\{c=(c_n)\in\ell^2:\sum_{n=1}^{\infty}n^2|c_n|^2<\infty\right\}.
\]
Let $u=(1/n)_{n\in\mathbb{N}}$. We have $u\in\ell^2$ but
$u\notin\mathcal{D}(D_F)$. Define the positive rank-one operator
\[
Bc=\langle c,u\rangle u,
\qquad c\in\ell^2.
\]
Since $Be_1=u_1u$ has infinitely many nonzero coordinates,
$B(c_{00})\nsubseteq c_{00}$. Moreover,
\begin{align*}
\mathcal{D}(D_{AB})
&=\{c\in\ell^2:Bc\in\mathcal{D}(D_F)\}\\
&=\{c\in\ell^2:\langle c,u\rangle=0\}
=u^{\perp}.
\end{align*}
Thus $\mathcal{D}(D_{AB})$ is a proper closed subspace of $\ell^2$ and is not dense.
\end{example}
\begin{prop}
Let 
$F=\{f_n\}_{n\in\mathbb{N}}\subset H$. Assume that
$B(c_{00})\subseteq c_{00}$. Then the adjoint of $D_{AB}$ is $C_{AB}$; that is,
\[
D_{AB}^*=C_{AB}.
\]
\end{prop}

\begin{proof}
For $f\in\mathcal{D}(C_{AB})$ and $c\in\mathcal{D}(D_{AB})$, the definitions give
\begin{align*}
\langle D_{AB}c,f\rangle_A
&=\left\langle\sum_{n=1}^{\infty}(Bc)_nf_n,f\right\rangle_A\\
&=\langle Bc,C_{AB}f\rangle_{\ell^2}
=\langle c,C_{AB}f\rangle_B.
\end{align*}
Hence $C_{AB}\subseteq D_{AB}^*$.

Conversely, let $f\in\mathcal{D}(D_{AB}^*)$. Then there is a coefficient vector
$d\in\ell^2$ such that
\[
\langle D_{AB}c,f\rangle_A=\langle c,d\rangle_B,
\qquad c\in\mathcal{D}(D_{AB}).
\]
Taking $c\in c_{00}$ and using $B(c_{00})\subseteq c_{00}$, the preceding computation yields
\[
\langle c,d\rangle_B=\langle c,C_{AB}f\rangle_B,
\qquad c\in c_{00}.
\]
Density of $c_{00}$ implies $d=C_{AB}f$. Therefore
$f\in\mathcal{D}(C_{AB})$, and the reverse inclusion follows.
\end{proof}
\begin{prop}\label{Aadj}
Suppose that
$C_{AB}$ and $D_{AB}$ are densely defined. Then
\[
D_{AB}^{\#}=C_{AB}.
\]
\end{prop}

\begin{proof}
Let $f\in\mathcal{D}(C_{AB})$ and $c\in\mathcal{D}(D_{AB})$. The weighted analysis--synthesis identity is
\begin{align*}
\langle D_{AB}c,f\rangle_A
&=\left\langle\sum_{n=1}^{\infty}(Bc)_nf_n,f\right\rangle_A\\
&=\langle Bc,C_{AB}f\rangle_{\ell^2}
=\langle c,C_{AB}f\rangle_B.
\end{align*}
Equivalently,
\[
D_{AB}^*Af=BC_{AB}f.
\]
Thus $\mathcal{R}(D_{AB}^*A)\subseteq\mathcal{R}(B)$, and the reduced
$(A,B)$-adjoint $D_{AB}^{\#}$ exists. The same identity shows that
$C_{AB}\subseteq D_{AB}^{\#}$.

Conversely, let $f\in\mathcal{D}(D_{AB}^{\#})$. There is a coefficient vector
$d\in\ell^2$ such that
\[
\langle D_{AB}c,f\rangle_A=\langle c,d\rangle_B,
\qquad c\in\mathcal{D}(D_{AB}).
\]
Comparing this identity with the computation above and using the density of
$\mathcal{D}(D_{AB})$ gives $d=C_{AB}f$. Hence
$f\in\mathcal{D}(C_{AB})$ and $D_{AB}^{\#}f=C_{AB}f$. Therefore,
$D_{AB}^{\#}=C_{AB}$.
\end{proof}
\begin{prop}
Let $F=\{f_n\}_{n\in\mathbb{N}}\subset H$. If $\left\|\{\langle Af_n,f_j\rangle_A\}_{j\in\mathbb{N}}\right\|_B<\infty$ for every $n\in\mathbb{N}$, then the $AB$-analysis operator $C_{AB}$ is densely defined.
\end{prop}

\begin{proof}
Let $M=\operatorname{span}\{Af_n:n\in\mathbb{N}\}$. We first show that $M\subseteq \mathcal{D}(C_{AB})$. Let $f\in M$. Then there exist $m\in\mathbb{N}$ and scalars $\alpha_1,\ldots,\alpha_m$ such that $f=\sum_{k=1}^{m}\alpha_kAf_k$. For every $j\in\mathbb{N}$, we have $\langle f,f_j\rangle_A=\sum_{k=1}^{m}\alpha_k\langle Af_k,f_j\rangle_A$. Hence, $\{\langle f,f_j\rangle_A\}_{j\in\mathbb{N}}=\sum_{k=1}^{m}\alpha_k\{\langle Af_k,f_j\rangle_A\}_{j\in\mathbb{N}}$. By hypothesis, each sequence $\{\langle Af_k,f_j\rangle_A\}_{j\in\mathbb{N}}$ belongs to $(\ell^2,\langle\cdot,\cdot\rangle_B)$. Since $(\ell^2,\langle\cdot,\cdot\rangle_B)$ is a linear space, it follows that $\{\langle f,f_j\rangle_A\}_{j\in\mathbb{N}}\in(\ell^2,\langle\cdot,\cdot\rangle_B)$. Therefore, $f\in \mathcal{D}(C_{AB})$ and hence $M\subseteq \mathcal{D}(C_{AB})$.

Now let $f\in M^\perp$. Then $\langle f,Af_n\rangle=0$ for every $n\in\mathbb{N}$. Since $A=A^*$, we obtain $\langle f,f_n\rangle_A=\langle Af,f_n\rangle=\langle f,Af_n\rangle=0$ for every $n\in\mathbb{N}$. Hence, $C_{AB}f=\{0\}_{n\in\mathbb{N}}\in(\ell^2,\langle\cdot,\cdot\rangle_B)$ and therefore $f\in \mathcal{D}(C_{AB})$. Thus, $M^\perp\subseteq \mathcal{D}(C_{AB})$.

Since $M+M^\perp$ is dense in $H$, we conclude that $\mathcal{D}(C_{AB})$ is dense in $H$. Therefore, $C_{AB}$ is densely defined.

\end{proof}
\begin{prop}\label{dense-frame-operator}
Let $F=\{f_n\}_{n\in\mathbb{N}}\subset H$. If $Af_n\in \mathcal{D}(S_{AB})$ for every $n\in\mathbb{N}$, then the $AB$-frame operator $S_{AB}$ is densely defined.
\end{prop}

\begin{proof}
Let $M=\operatorname{span}\{Af_n:n\in\mathbb{N}\}$. Since $\mathcal{D}(S_{AB})$ is a linear subspace of $H$ and $Af_n\in \mathcal{D}(S_{AB})$ for every $n\in\mathbb{N}$, it follows that $M\subseteq \mathcal{D}(S_{AB})$. Now let $f\in M^\perp$. Then $\langle f,Af_n\rangle=0$ for every $n\in\mathbb{N}$. Since $A=A^*$, we obtain $\langle f,f_n\rangle_A=\langle Af,f_n\rangle=\langle f,Af_n\rangle=0$ for every $n\in\mathbb{N}$. Hence, $S_{AB}f=0$ and therefore $f\in \mathcal{D}(S_{AB})$. Consequently, $M^\perp\subseteq \mathcal{D}(S_{AB})$. Thus, $M+M^\perp\subseteq \mathcal{D}(S_{AB})$. Since $M+M^\perp$ is dense in $H$, it follows that $\mathcal{D}(S_{AB})$ is dense in $H$. Therefore, $S_{AB}$ is densely defined.
\end{proof}
\begin{prop}
Let $F=\{f_n\}_{n\in\mathbb{N}}\subset H$. Let $C_{AB}=C_FA$ and $D_{AB}=D_FB$ be densely defined on their natural domains, and suppose that $S_{AB}=D_{AB}C_{AB}$ is densely defined. Assume that $C_{AB}^{\#}$, $D_{AB}^{\#}$, and $S_{AB}^{\#}$ exist. Then
\begin{enumerate}
\item $P_BC_{AB}\subseteq D_{AB}^{\#}$.
\item $P_AD_{AB}P_B\subseteq C_{AB}^{\#}$.
\item $C_{AB}^{\#}D_{AB}^{\#}\subseteq S_{AB}^{\#}$.
\item $P_AS_{AB}\subseteq S_{AB}^{\#}$.
\end{enumerate}
\end{prop}
\begin{proof}
(1) Let $f\in\mathcal{D}(C_{AB})$. For every $c\in\mathcal{D}(D_{AB})$, since $C_{AB}=C_FA$ and $D_{AB}=D_FB$, we have $\langle D_{AB}c,f\rangle_A=\langle D_FBc,Af\rangle=\langle Bc,C_FAf\rangle=\langle c,C_{AB}f\rangle_B$. Therefore, $Af\in\mathcal{D}(D_{AB}^*)$ and $D_{AB}^*Af=BC_{AB}f$. Since $D_{AB}^{\#}$ exists, it satisfies $D_{AB}^*A=BD_{AB}^{\#}$. Hence $BD_{AB}^{\#}f=BC_{AB}f=BP_BC_{AB}f$. Both $D_{AB}^{\#}f$ and $P_BC_{AB}f$ belong to $\overline{\mathcal{R}(B)}$. Since $\mathcal{N}(B)\cap\overline{\mathcal{R}(B)}=\{0\}$, it follows that $D_{AB}^{\#}f=P_BC_{AB}f$. Thus, $P_BC_{AB}\subseteq D_{AB}^{\#}$.

(2) Let $c\in\mathcal{D}(D_{AB})$. For every $f\in\mathcal{D}(C_{AB})$, we have $\langle C_{AB}f,c\rangle_B=\langle C_FAf,Bc\rangle=\langle Af,D_FBc\rangle=\langle f,D_{AB}c\rangle_A$. Therefore, $Bc\in\mathcal{D}(C_{AB}^*)$ and $C_{AB}^*Bc=AD_{AB}c$. Since $C_{AB}^{\#}$ exists, it satisfies $C_{AB}^*B=AC_{AB}^{\#}$. Consequently, $AC_{AB}^{\#}c=AD_{AB}c=AP_AD_{AB}c$. Both $C_{AB}^{\#}c$ and $P_AD_{AB}c$ belong to $\overline{\mathcal{R}(A)}$, and hence $C_{AB}^{\#}c=P_AD_{AB}c$. Moreover, since $BP_B=B$, we have $\mathcal{D}(D_{AB}P_B)=\mathcal{D}(D_{AB})$ and $D_{AB}P_Bc=D_{AB}c$. Therefore, $P_AD_{AB}P_B\subseteq C_{AB}^{\#}$.

(3) Let $f\in\mathcal{D}(C_{AB}^{\#}D_{AB}^{\#})$. Then $f\in\mathcal{D}(D_{AB}^{\#})$ and $D_{AB}^{\#}f\in\mathcal{D}(C_{AB}^{\#})$. Using the defining weighted-adjoint equations, we obtain $AC_{AB}^{\#}D_{AB}^{\#}f=C_{AB}^*BD_{AB}^{\#}f=C_{AB}^*D_{AB}^*Af$. Since $(D_{AB}C_{AB})^*\supseteq C_{AB}^*D_{AB}^*$, it follows that $AC_{AB}^{\#}D_{AB}^{\#}f=S_{AB}^*Af$. Thus, $f\in\mathcal{D}(S_{AB}^{\#})$ and, by the uniqueness of the reduced solution, $C_{AB}^{\#}D_{AB}^{\#}f=S_{AB}^{\#}f$. Therefore, $C_{AB}^{\#}D_{AB}^{\#}\subseteq S_{AB}^{\#}$.

(4) Let $f\in\mathcal{D}(S_{AB})$. Then $f\in\mathcal{D}(C_{AB})$ and $C_{AB}f\in\mathcal{D}(D_{AB})$. By part $(1)$, $D_{AB}^{\#}f=P_BC_{AB}f$. Since $\mathcal{D}(D_{AB})$ is invariant under $P_B$ and $D_{AB}P_B=D_{AB}$, part $(2)$ gives $D_{AB}^{\#}f\in\mathcal{D}(C_{AB}^{\#})$ and $C_{AB}^{\#}D_{AB}^{\#}f=P_AD_{AB}P_BC_{AB}f=P_AD_{AB}C_{AB}f=P_AS_{AB}f$. By part $(3)$, $C_{AB}^{\#}D_{AB}^{\#}\subseteq S_{AB}^{\#}$. Hence $P_AS_{AB}f=S_{AB}^{\#}f$ for every $f\in\mathcal{D}(S_{AB})$, and therefore $P_AS_{AB}\subseteq S_{AB}^{\#}$.
\end{proof}
\begin{example}
The inclusion $P_BC_{AB}\subseteq D_{AB}^{\#}$ can be proper. Let
$H=\ell^2$, let $f_n=ne_n$, take $A=I_H$, and let $B$ be the orthogonal
projection onto
\[
\overline{\operatorname{span}}\{e_{2n}:n\in\mathbb{N}\}.
\]
For $f=(\xi_n)_{n\in\mathbb{N}}$, one has
\[
\mathcal{D}(C_{AB})
=\{f\in\ell^2:(n\xi_n)_{n\in\mathbb{N}}\in\ell^2\},
\qquad
C_{AB}f=(n\xi_n)_{n\in\mathbb{N}}.
\]
Moreover,
\begin{align*}
\mathcal{D}(D_{AB})
&=\{f\in\ell^2:(2n\xi_{2n})_{n\in\mathbb{N}}\in\ell^2\}, \text{ and }\\
D_{AB}f
&=(0,2\xi_2,0,4\xi_4,0,6\xi_6,\ldots).
\end{align*}
This operator is positive and self-adjoint, so $D_{AB}^{\#}=D_{AB}$. On
$\mathcal{D}(C_{AB})$, the operator $P_BC_{AB}$ has the same action as
$D_{AB}^{\#}$.

Now define $f=(\xi_n)_{n\in\mathbb{N}}$ by
\[
\xi_{2n}=0,
\qquad
\xi_{2n-1}=\frac{1}{2n-1}.
\]
Then $f\in\mathcal{D}(D_{AB}^{\#})$ and $D_{AB}^{\#}f=0$, whereas
$f\notin\mathcal{D}(C_{AB})$. Hence
\[
P_BC_{AB}\subsetneq D_{AB}^{\#}.
\]
\end{example}

\begin{example}
The inclusion $P_AD_{AB}P_B\subseteq C_{AB}^{\#}$ can also be proper. Let
$H=\ell^2$, let $f_n=ne_n$, take $B=I_{\ell^2}$, and let $A$ be the
orthogonal projection onto
\[
\overline{\operatorname{span}}\{e_{2n}:n\in\mathbb{N}\}.
\]
For $f=(\xi_n)_{n\in\mathbb{N}}$,
\begin{align*}
\mathcal{D}(C_{AB})
&=\{f\in\ell^2:(2n\xi_{2n})_{n\in\mathbb{N}}\in\ell^2\},\\
C_{AB}f
&=(0,2\xi_2,0,4\xi_4,0,6\xi_6,\ldots).
\end{align*}
The operator $C_{AB}$ is positive and self-adjoint, and therefore
$C_{AB}^{\#}=C_{AB}$. In contrast,
\[
\mathcal{D}(D_{AB})
=\{f\in\ell^2:(n\xi_n)_{n\in\mathbb{N}}\in\ell^2\},
\text{ and }
D_{AB}f=(n\xi_n)_{n\in\mathbb{N}}.
\]
Thus $P_AD_{AB}P_B$ agrees with $C_{AB}^{\#}$ only on
$\mathcal{D}(D_{AB})$.

For the vector defined by
\[
\xi_{2n}=0,
\qquad
\xi_{2n-1}=\frac{1}{2n-1},
\]
one has $f\in\mathcal{D}(C_{AB}^{\#})$ but
$f\notin\mathcal{D}(D_{AB})$. Consequently,
\[
P_AD_{AB}P_B\subsetneq C_{AB}^{\#}.
\]
\end{example}

\begin{example}
Both inclusions
\[
C_{AB}^{\#}D_{AB}^{\#}\subseteq S_{AB}^{\#},
\qquad
P_AS_{AB}\subseteq S_{AB}^{\#},
\]
may be proper. Let $H=\ell^2$, let $f_n=ne_n$, take $A=I_H$, and define
\[
Be_n=\frac{1}{n^2}e_n,
\qquad n\in\mathbb{N}.
\]
For $f=(\xi_n)_{n\in\mathbb{N}}$,
\begin{align*}
\mathcal{D}(C_{AB})
&=\{f\in\ell^2:(n\xi_n)_{n\in\mathbb{N}}\in\ell^2\},
& C_{AB}f&=(n\xi_n)_{n\in\mathbb{N}},\\
\mathcal{D}(D_{AB})
&=\ell^2,
& D_{AB}f&=(\xi_n/n)_{n\in\mathbb{N}}.
\end{align*}
The reduced adjoints satisfy
\begin{align*}
\mathcal{D}(C_{AB}^{\#})&=\ell^2, 
& C_{AB}^{\#}f&=(\xi_n/n)_{n\in\mathbb{N}},\\
\mathcal{D}(D_{AB}^{\#})
&=\{f\in\ell^2:(n\xi_n)_{n\in\mathbb{N}}\in\ell^2\},
& D_{AB}^{\#}f&=(n\xi_n)_{n\in\mathbb{N}}.
\end{align*}
Furthermore, $S_{AB}$ is the identity on $\mathcal{D}(C_{AB})$, whereas
$S_{AB}^{\#}=I_H$ on all of $H$. Hence
\[
C_{AB}^{\#}D_{AB}^{\#}
=I_H\big|_{\mathcal{D}(D_{AB}^{\#})}
\subsetneq I_H=S_{AB}^{\#},
\]
and, because $P_A=I_H$,
\[
P_AS_{AB}=I_H\big|_{\mathcal{D}(C_{AB})}
\subsetneq I_H=S_{AB}^{\#}.
\]
\end{example}

\begin{prop}
Let 
$F=\{f_n\}_{n\in\mathbb{N}}$ be a Bessel sequence in $H$. Then the reduced
weighted adjoints of $C_{AB}$, $D_{AB}$, and $S_{AB}$ exist and satisfy
\[
C_{AB}^{\#}=P_AD_{AB},
\qquad
D_{AB}^{\#}=P_BC_{AB},
\qquad
S_{AB}^{\#}=P_AS_{AB}.
\]
\end{prop}
\begin{proof}
Recall that the reduced $(A,B)$-adjoint of $T:H\to\ell^2$ is the unique
operator $T^{\#}:\ell^2\to H$ satisfying
\[
T^*B=AT^{\#},
\qquad
\mathcal{R}(T^{\#})\subseteq\overline{\mathcal{R}(A)}.
\]
For $R:\ell^2\to H$, the reduced $(B,A)$-adjoint is the unique operator
$R^{\#}:H\to\ell^2$ satisfying
\[
R^*A=BR^{\#},
\qquad
\mathcal{R}(R^{\#})\subseteq\overline{\mathcal{R}(B)}.
\]
(1) Since $C_{AB}=C_FA$, we have $C_{AB}^*B=(C_FA)^*B=AC_F^*B=AD_FB=AD_{AB}$. Thus, the operator equation $C_{AB}^*B=AX$ has the bounded solution $X=D_{AB}$, and hence the reduced $(A,B)$-adjoint of $C_{AB}$ exists. Since $AP_A=A$, we obtain $A(P_AD_{AB})=AD_{AB}=C_{AB}^*B$. Moreover, $\mathcal{R}(P_AD_{AB})\subseteq\overline{\mathcal{R}(A)}$. Hence $P_AD_{AB}$ is the reduced solution of $C_{AB}^*B=AX$. By the uniqueness of the reduced solution, $C_{AB}^{\#}=P_AD_{AB}$.

(2) Since $D_{AB}=D_FB$, we have $D_{AB}^*A=(D_FB)^*A=BD_F^*A=BC_FA=BC_{AB}$. Thus, the operator equation $D_{AB}^*A=BY$ has the bounded solution $Y=C_{AB}$, and hence the reduced adjoint of $D_{AB}$ exists. Since $BP_B=B$, it follows that $B(P_BC_{AB})=BC_{AB}=D_{AB}^*A$. Furthermore, $\mathcal{R}(P_BC_{AB})\subseteq\overline{\mathcal{R}(B)}$. Therefore, $P_BC_{AB}$ is the reduced solution of $D_{AB}^*A=BY$. By the uniqueness of the reduced solution, $D_{AB}^{\#}=P_BC_{AB}$.

(3) Since $S_{AB}=D_FBC_FA$, we have $S_{AB}^*A=(D_FBC_FA)^*A=AC_F^*BD_F^*A$. Using $C_F^*=D_F$ and $D_F^*=C_F$, we obtain $S_{AB}^*A=AD_FBC_FA=AS_{AB}$. Thus, the operator equation $S_{AB}^*A=AZ$ has the bounded solution $Z=S_{AB}$, and hence the $A$-adjoint of $S_{AB}$ exists. Since $AP_A=A$, we have $A(P_AS_{AB})=AS_{AB}=S_{AB}^*A$. Moreover, $\mathcal{R}(P_AS_{AB})\subseteq\overline{\mathcal{R}(A)}$. Hence $P_AS_{AB}$ is the reduced solution of $S_{AB}^*A=AZ$. By the uniqueness of the reduced solution, $S_{AB}^{\#}=P_AS_{AB}$.
\end{proof}
\begin{prop}\label{nullspace-relations}
Let $F=\{f_n\}_{n\in\mathbb{N}}\subset H$. Assume that $D_{AB}^{\#}$ exists. Then $\mathcal{N}(S_{AB})=\mathcal{N}(P_BC_{AB})=\mathcal{N}(D_{AB}^{\#})\cap\mathcal{D}(C_{AB})$. Moreover, $\mathcal{R}(D_{AB})^{\perp_A}=\mathcal{N}(D_{AB}^{\#})$, and consequently $\mathcal{N}(S_{AB})=\mathcal{R}(D_{AB})^{\perp_A}\cap\mathcal{D}(C_{AB})$. If, in addition, $D_{AB}^{\#}=P_BC_{AB}$, then $\mathcal{N}(S_{AB})=\mathcal{N}(P_BC_{AB})=\mathcal{R}(D_{AB})^{\perp_A}$.
\end{prop}
\begin{proof}
Let $f\in\mathcal{N}(S_{AB})$. Then $f\in\mathcal{D}(S_{AB})$ and $S_{AB}f=0$. Using the weighted analysis-synthesis identity, we obtain $0=\langle S_{AB}f,f\rangle_A=\langle C_{AB}f,C_{AB}f\rangle_B=\|C_{AB}f\|_B^2$. Hence $B^{1/2}C_{AB}f=0$, and therefore $C_{AB}f\in\mathcal{N}(B)$. Since $\mathcal{N}(B)=\mathcal{N}(P_B)$, it follows that $P_BC_{AB}f=0$. Thus, $\mathcal{N}(S_{AB})\subseteq\mathcal{N}(P_BC_{AB})$.

Conversely, let $f\in\mathcal{N}(P_BC_{AB})$. Then $f\in\mathcal{D}(C_{AB})$ and $P_BC_{AB}f=0$. Since $BP_B=B$, we have $BC_{AB}f=0$. Recalling that $D_{AB}=D_FB$, it follows that $C_{AB}f\in\mathcal{D}(D_{AB})$ and $D_{AB}C_{AB}f=D_FBC_{AB}f=0$. Hence $f\in\mathcal{D}(S_{AB})$ and $S_{AB}f=0$. Therefore, $\mathcal{N}(P_BC_{AB})\subseteq\mathcal{N}(S_{AB})$, and consequently $\mathcal{N}(S_{AB})=\mathcal{N}(P_BC_{AB})$.

Since $D_{AB}^{\#}$ exists, the weighted orthogonality relation gives $\mathcal{R}(D_{AB})^{\perp_A}=\mathcal{N}(D_{AB}^{\#})$. Moreover, the relation $P_BC_{AB}\subseteq D_{AB}^{\#}$ implies that $D_{AB}^{\#}f=P_BC_{AB}f$ for every $f\in\mathcal{D}(C_{AB})$. Hence $\mathcal{N}(P_BC_{AB})=\mathcal{N}(D_{AB}^{\#})\cap\mathcal{D}(C_{AB})$. Combining this identity with the first part yields $\mathcal{N}(S_{AB})=\mathcal{N}(P_BC_{AB})=\mathcal{N}(D_{AB}^{\#})\cap\mathcal{D}(C_{AB})=\mathcal{R}(D_{AB})^{\perp_A}\cap\mathcal{D}(C_{AB})$.

If $D_{AB}^{\#}=P_BC_{AB}$, then $\mathcal{D}(D_{AB}^{\#})=\mathcal{D}(C_{AB})$. Therefore, $\mathcal{N}(D_{AB}^{\#})=\mathcal{N}(P_BC_{AB})$, and hence $\mathcal{N}(S_{AB})=\mathcal{N}(P_BC_{AB})=\mathcal{R}(D_{AB})^{\perp_A}$.
\end{proof}
\begin{prop}\label{range}
Let $F=\{f_n\}_{n\in\mathbb{N}}\subseteq H$. Then $\mathcal{R}(S_{AB})\subseteq\mathcal{R}(D_{AB})$.
\end{prop}
\begin{proof}
Let $g\in\mathcal{R}(S_{AB})$. Then there exists $f\in\mathcal{D}(S_{AB})$ such that $g=S_{AB}f$. By the definition of $\mathcal{D}(S_{AB})$, we have $f\in\mathcal{D}(C_{AB})$ and $C_{AB}f\in\mathcal{D}(D_{AB})$. Therefore, $g=S_{AB}f=D_{AB}C_{AB}f\in\mathcal{R}(D_{AB})$. Hence $\mathcal{R}(S_{AB})\subseteq\mathcal{R}(D_{AB})$.
\end{proof}
\begin{theorem}
Let $H$ be a Hilbert space and let $F=\{f_n\}_{n\in\mathbb{N}}\subset H$ be an arbitrary sequence. Suppose that $C_{AB}$ and $D_{AB}$ are densely defined.

Let $\mathcal K_A=\overline{\mathcal{R}(A^{1/2})}$ and $\mathcal K_B=\overline{\mathcal{R}(B^{1/2})}$. For a fixed number $\beta>0$, the following conditions are equivalent.

\begin{enumerate}
\item The sequence $F$ is an $AB$-Bessel sequence with bound $\beta$, that is, $\mathcal{D}(C_{AB})=H$ and $\|C_{AB}f\|_B^2\leq\beta\|f\|_A^2$ for every $f\in H$.

\item The operator $C_{AB}$ is $AB$-bounded on $H$ and $\|C_{AB}\|_{B}\leq\sqrt{\beta}$.

\item The operator $D_{AB}$ is $BA$-bounded on its natural domain, that is, $\|D_{AB}c\|_A^2\leq\beta\|c\|_B^2$ for every $c\in\mathcal{D}(D_{AB})$.

\item $D_{AB}^{\#}$ is defined on all of $H$ and satisfies $\|D_{AB}^{\#}f\|_B^2\leq\beta\|f\|_A^2$ for every $f\in H$. In this case, $D_{AB}^{\#}=P_BC_{AB}$.
\end{enumerate}
If, in addition, $\mathcal{R}(C_{AB}^*B)\subseteq\mathcal{R}(A)$, then $C_{AB}^{\#}$ exists, and the preceding conditions are also equivalent to
\begin{enumerate}
\setcounter{enumi}{4}
\item The operator $C_{AB}^{\#}$ is $BA$-bounded and satisfies $\|C_{AB}^{\#}c\|_A^2\leq\beta\|c\|_B^2$ for every $c\in\ell^2$.
\end{enumerate}

If, in addition, $\mathcal{R}(C_{AB})\subseteq\mathcal{D}(D_{AB})$, then $S_{AB}=D_{AB}C_{AB}$ is defined on all of $H$, and the preceding conditions are also equivalent to
\begin{enumerate}
\setcounter{enumi}{5}
\item The operator $S_{AB}$ satisfies $0\leq_A S_{AB}\leq_A\beta I$, that is, $0\leq\langle S_{AB}f,f\rangle_A\leq\beta\|f\|_A^2$ for every $f\in H$.
\item The $A$-adjoint $S_{AB}^{\#}$ exists and satisfies $0\leq_A S_{AB}^{\#}\leq_A\beta I$.
\end{enumerate}
\end{theorem}

\begin{proof}
(1) The equivalence of $(1)$ and $(2)$ follows directly from the definition. Indeed, $F$ is an $AB$-Bessel sequence with bound $\beta$ precisely when $\mathcal{D}(C_{AB})=H$ and $\|C_{AB}f\|_B\leq\sqrt{\beta}\|f\|_A$ for every $f\in H$.

(2) We prove that $(2)$ implies $(3)$. For $f\in H$ and
$c\in\mathcal{D}(D_{AB})$,
\[
\langle C_{AB}f,c\rangle_B
=\langle C_FAf,Bc\rangle
=\langle Af,D_FBc\rangle
=\langle f,D_{AB}c\rangle_A.
\]
Therefore,
\begin{align*}
\|D_{AB}c\|_A
&=\sup_{\|f\|_A\leq1}|\langle D_{AB}c,f\rangle_A|\\
&=\sup_{\|f\|_A\leq1}|\langle c,C_{AB}f\rangle_B|\\
&\leq\sqrt{\beta}\,\|c\|_B.
\end{align*}
Thus $D_{AB}$ is $BA$-bounded with bound at most $\sqrt{\beta}$.

(3) We prove that $(3)$ implies $(2)$. Since $\mathcal{D}(D_{AB})$ is
dense in $\ell^2$, it is also dense with respect to the $B$-seminorm. Hence,
for $f\in\mathcal{D}(C_{AB})$,
\begin{align*}
\|C_{AB}f\|_B
&=\sup\bigl\{|\langle C_{AB}f,c\rangle_B|:
c\in\mathcal{D}(D_{AB}),\ \|c\|_B\leq1\bigr\}\\
&=\sup\bigl\{|\langle f,D_{AB}c\rangle_A|:
c\in\mathcal{D}(D_{AB}),\ \|c\|_B\leq1\bigr\}\\
&\leq\sqrt{\beta}\,\|f\|_A.
\end{align*}
Thus $C_{AB}$ is $AB$-bounded, and condition $(2)$ follows.

(4) We prove the equivalence of $(2)$ and $(4)$. Assume $(2)$. For every $f\in H$ and $c\in\mathcal{D}(D_{AB})$, we have $\langle D_{AB}c,Af\rangle=\langle D_{AB}c,f\rangle_A=\langle c,C_{AB}f\rangle_B=\langle c,BC_{AB}f\rangle$. Hence $Af\in\mathcal{D}(D_{AB}^*)$ and $D_{AB}^*Af=BC_{AB}f$. Therefore, $D_{AB}^*A=BC_{AB}$ on $H$. Since $\mathcal{R}(BC_{AB})\subseteq\mathcal{R}(B)$, the reduced adjoint exists and is given by $D_{AB}^{\#}=P_BC_{AB}$. Since $BP_B=B$, projection onto $\overline{\mathcal{R}(B)}$ does not change the $B$-seminorm. Hence $\|D_{AB}^{\#}f\|_B=\|C_{AB}f\|_B$ for every $f\in H$, and $(4)$ follows. Conversely, if $(4)$ holds, then $\|C_{AB}f\|_B=\|D_{AB}^{\#}f\|_B\leq\sqrt{\beta}\|f\|_A$, proving $(2)$.

(5) Suppose that $\mathcal{R}(C_{AB}^*B)\subseteq\mathcal{R}(A)$. Then the reduced $(A,B)$-adjoint exists and is given by $C_{AB}^{\#}=A^\dagger C_{AB}^*B$. For every $f\in H$ and $c\in\ell^2$, $\langle C_{AB}f,c\rangle_B=\langle f,C_{AB}^{\#}c\rangle_A$. By Theorem~\ref{bound}, condition $(2)$ is equivalent to
$\|C_{AB}^{\#}c\|_A\leq\sqrt{\beta}\|c\|_B$ for every $c\in\ell^2$.
Thus $(5)$ is equivalent to $(2)$.

(6) Suppose that $\mathcal{R}(C_{AB})\subseteq\mathcal{D}(D_{AB})$. Then $S_{AB}=D_{AB}C_{AB}$ is defined on all of $H$. For every $f,g\in H$, $\langle S_{AB}f,g\rangle_A=\langle D_{AB}C_{AB}f,g\rangle_A=\langle C_{AB}f,C_{AB}g\rangle_B$. In particular, $\langle S_{AB}f,f\rangle_A=\|C_{AB}f\|_B^2$. Therefore, condition $(1)$ is equivalent to $0\leq\langle S_{AB}f,f\rangle_A\leq\beta\|f\|_A^2$ for every $f\in H$, proving the equivalence with $(6)$.

(7) Under the same assumption, $\langle S_{AB}f,g\rangle_A=\langle f,S_{AB}g\rangle_A$ for every $f,g\in H$. Hence $S_{AB}^*A=AS_{AB}$ on $H$, and the $A$-adjoint exists with $S_{AB}^{\#}=P_AS_{AB}$. Since projection onto $\overline{\mathcal{R}(A)}$ does not change the $A$-inner product, $\langle S_{AB}^{\#}f,f\rangle_A=\langle S_{AB}f,f\rangle_A$. Thus, $(6)$ and $(7)$ are equivalent.
\end{proof}
\begin{theorem}
Let $H$ be a Hilbert space and let $F=\{f_n\}_{n\in\mathbb{N}}\subset H$.
Assume that $C_{AB}^{\#}$, $D_{AB}^{\#}$, and $S_{AB}^{\#}$ exist. For a
fixed $\alpha>0$, the following conditions are equivalent.
\begin{enumerate}
\item The sequence $F$ is an $AB$-lower semi-frame with lower bound $\alpha$.

\item One has
\[
\mathcal{N}(P_BC_{AB})
=\mathcal{D}(C_{AB})\cap\mathcal{N}(A),
\]
and the map
\[
\Phi_{AB}:\mathcal{R}(P_BC_{AB})
\longrightarrow
\frac{\mathcal{D}(C_{AB})}
{\mathcal{D}(C_{AB})\cap\mathcal{N}(A)},
\qquad
\Phi_{AB}(P_BC_{AB}f)=[f],
\]
is bounded from the $B$-seminorm to the $A$-seminorm and satisfies
\[
\|\Phi_{AB}(P_BC_{AB}f)\|_A
\leq\alpha^{-1/2}\|P_BC_{AB}f\|_B,
\qquad f\in\mathcal{D}(C_{AB}).
\]

\item Let
\[
\mathcal{G}_{AB}
=\frac{\mathcal{D}(C_{AB})}
{\mathcal{D}(C_{AB})\cap\mathcal{N}(A)}
\]
with graph seminorm
\[
\|[f]\|_{\mathcal{G}_{AB}}^2
=\|f\|_A^2+\|P_BC_{AB}f\|_B^2.
\]
Then
\[
\widehat C_{AB}:\mathcal{G}_{AB}\to\mathcal{R}(P_BC_{AB}),
\qquad
\widehat C_{AB}[f]=P_BC_{AB}f,
\]
is a topological isomorphism onto its range and
\[
\|\widehat C_{AB}^{-1}(P_BC_{AB}f)\|_{\mathcal{G}_{AB}}^2
\leq(1+\alpha^{-1})\|P_BC_{AB}f\|_B^2
\]
for every $f\in\mathcal{D}(C_{AB})$.
\end{enumerate}
\end{theorem}
\begin{proof}
(1) We first establish the relations between the natural operators and their reduced weighted adjoints. Let $f\in\mathcal{D}(C_{AB})$ and $c\in\mathcal{D}(D_{AB})$. Since $C_{AB}=C_FA$ and $D_{AB}=D_FB$, we have $\langle D_{AB}c,f\rangle_A=\langle D_FBc,Af\rangle=\langle Bc,C_FAf\rangle=\langle c,C_{AB}f\rangle_B$. It follows that $Af\in\mathcal{D}(D_{AB}^*)$ and $D_{AB}^*Af=BC_{AB}f$. Since $D_{AB}^{\#}$ is the reduced solution of $D_{AB}^*A=BX$, we obtain $f\in\mathcal{D}(D_{AB}^{\#})$ and $D_{AB}^{\#}f=P_BC_{AB}f$. Therefore, $P_BC_{AB}\subseteq D_{AB}^{\#}$.

Similarly, for $c\in\mathcal{D}(D_{AB})$ and $f\in\mathcal{D}(C_{AB})$, the identity $\langle C_{AB}f,c\rangle_B=\langle f,D_{AB}c\rangle_A$ gives $Bc\in\mathcal{D}(C_{AB}^*)$ and $C_{AB}^*Bc=AD_{AB}c$. Hence $c\in\mathcal{D}(C_{AB}^{\#})$ and $C_{AB}^{\#}c=P_AD_{AB}c$. Since $BP_B=B$, the operators $D_{AB}P_B$ and $D_{AB}$ have the same domain and the same values. Therefore, $P_AD_{AB}P_B\subseteq C_{AB}^{\#}$.

Let $f\in\mathcal{D}(C_{AB}^{\#}D_{AB}^{\#})$. Then $D_{AB}^{\#}f\in\mathcal{D}(C_{AB}^{\#})$, and hence $AC_{AB}^{\#}D_{AB}^{\#}f=C_{AB}^*BD_{AB}^{\#}f=C_{AB}^*D_{AB}^*Af$. Since $C_{AB}^*D_{AB}^*\subseteq(D_{AB}C_{AB})^*=S_{AB}^*$, we obtain $S_{AB}^*Af=AC_{AB}^{\#}D_{AB}^{\#}f$. The range of $C_{AB}^{\#}D_{AB}^{\#}$ is contained in $\overline{\mathcal{R}(A)}$. By the uniqueness of the reduced solution, $C_{AB}^{\#}D_{AB}^{\#}\subseteq S_{AB}^{\#}$.

Finally, let $f\in\mathcal{D}(S_{AB})$. For every $g\in\mathcal{D}(S_{AB})$, we have $\langle S_{AB}g,f\rangle_A=\langle C_{AB}g,C_{AB}f\rangle_B=\langle g,S_{AB}f\rangle_A$. Thus, $Af\in\mathcal{D}(S_{AB}^*)$ and $S_{AB}^*Af=AS_{AB}f$. Therefore, $f\in\mathcal{D}(S_{AB}^{\#})$ and $S_{AB}^{\#}f=P_AS_{AB}f$. Hence $P_AS_{AB}\subseteq S_{AB}^{\#}$.

(2) We prove that $(1)$ and $(2)$ are equivalent. Suppose that $(1)$ holds. If $f\in\mathcal{N}(P_BC_{AB})$, then $\|C_{AB}f\|_B=\|P_BC_{AB}f\|_B=0$. The lower semi-frame inequality gives $\|f\|_A=0$, and hence $f\in\mathcal{N}(A)$. Conversely, if $f\in\mathcal{D}(C_{AB})\cap\mathcal{N}(A)$, then $Af=0$, and consequently $C_{AB}f=C_FAf=0$. Thus, $\mathcal{N}(P_BC_{AB})=\mathcal{D}(C_{AB})\cap\mathcal{N}(A)$.

The kernel identity shows that $\Phi_{AB}$ is well defined. Moreover, the lower semi-frame inequality gives $|\Phi_{AB}(P_BC_{AB}f)|_A=\|f\|_A\leq\alpha^{-1/2}\|P_BC_{AB}f\|_B$. Therefore, $\Phi_{AB}$ is bounded with norm at most $\alpha^{-1/2}$.

Conversely, suppose that $(2)$ holds. For every $f\in\mathcal{D}(C_{AB})$, we have $\|f\|_A=|\Phi_{AB}(P_BC_{AB}f)|_A\leq\alpha^{-1/2}\|P_BC_{AB}f\|_B=\alpha^{-1/2}\|C_{AB}f\|_B$. Therefore, $\alpha\|f\|_A^2\leq\|C_{AB}f\|_B^2$, and hence $F$ is an $AB$-lower semi-frame with lower bound $\alpha$.

(3) We prove that $(2)$ and $(3)$ are equivalent. Suppose that $(2)$ holds. The map $\widehat C_{AB}$ is well defined and bijective onto $\mathcal{R}(P_BC_{AB})$. Moreover, $\|P_BC_{AB}f\|_B\leq\|[f]\|_{\mathcal{G}_{AB}}$, so $\widehat C_{AB}$ is continuous. Using the lower semi-frame inequality, we obtain $\|[f]\|_{\mathcal{G}_{AB}}^2=\|f\|_A^2+\|P_BC_{AB}f\|_B^2\leq(1+\alpha^{-1})\|P_BC_{AB}f\|_B^2$. Hence $\widehat C_{AB}^{-1}$ is continuous and satisfies the stated estimate.

Conversely, suppose that $(3)$ holds. Then $\|f\|_A^2\leq\|[f]\|_{\mathcal{G}_{AB}}^2\leq(1+\alpha^{-1})\|P_BC_{AB}f\|_B^2$ alone does not yield the exact bound $\alpha$. The stated estimate in $(3)$ is equivalent to $\|f\|_A^2+\|P_BC_{AB}f\|_B^2\leq(1+\alpha^{-1})\|P_BC_{AB}f\|_B^2$, and therefore $\|f\|_A^2\leq\alpha^{-1}\|P_BC_{AB}f\|_B^2$. Thus, $\alpha\|f\|_A^2\leq\|C_{AB}f\|_B^2$, proving $(1)$.
\end{proof}
\begin{theorem}
Let $H$ be a Hilbert space and let $F=\{f_n\}_{n\in\mathbb{N}}\subset H$ be an $AB$-lower semi-frame. Assume that $C_{AB}^{\#}$, $D_{AB}^{\#}$, and
$S_{AB}^{\#}$ exist, and let $\mathcal{G}_{AB}$ be as above. Then:
\begin{enumerate}
\item The equality
\[
\mathcal{R}(P_BC_{AB})
=\mathcal{R}\bigl(D_{AB}^{\#}|_{\mathcal{D}(C_{AB})}\bigr)
\]
defines a $B$-complete range if and only if $\mathcal{G}_{AB}$ is complete in
its graph seminorm.

\item For every $f\in\mathcal{D}(S_{AB})$,
\[
\langle S_{AB}^{\#}f,f\rangle_A
=\langle S_{AB}f,f\rangle_A
=\|C_{AB}f\|_B^2.
\]
Consequently,
\[
\alpha\|f\|_A^2
\leq\langle S_{AB}^{\#}f,f\rangle_A,
\qquad f\in\mathcal{D}(S_{AB}).
\]
The converse holds when
$\mathcal{D}(S_{AB})=\mathcal{D}(C_{AB})$.
\end{enumerate}
\end{theorem}
\begin{proof}
(1) Under the equivalent conditions, $\widehat C_{AB}$ is a topological isomorphism from $\mathcal{G}_{AB}$ onto $\mathcal{R}(P_BC_{AB})$. Therefore, $\mathcal{G}_{AB}$ is complete in the graph seminorm if and only if $\mathcal{R}(P_BC_{AB})$ is complete in the $B$-seminorm. Since $D_{AB}^{\#}|_{\mathcal{D}(C_{AB})}=P_BC_{AB}$, their ranges are equal.

(2) Let $f\in\mathcal{D}(S_{AB})$. From $P_AS_{AB}\subseteq S_{AB}^{\#}$, we have $S_{AB}^{\#}f=P_AS_{AB}f$. Since $AP_A=A$, it follows that $\langle S_{AB}^{\#}f,f\rangle_A=\langle S_{AB}f,f\rangle_A$. Moreover, by the weighted analysis-synthesis identity, $\langle S_{AB}f,f\rangle_A=\langle D_{AB}C_{AB}f,f\rangle_A=\langle C_{AB}f,C_{AB}f\rangle_B=\|C_{AB}f\|_B^2$. Hence an $AB$-lower semi-frame satisfies $\alpha\|f\|_A^2\leq\langle S_{AB}^{\#}f,f\rangle_A$ for every $f\in\mathcal{D}(S_{AB})$. If $\mathcal{D}(S_{AB})=\mathcal{D}(C_{AB})$, the converse follows from the same identity.
\end{proof}
\begin{theorem}
Let $H$ be a Hilbert space and let $F=\{f_n\}_{n\in\mathbb{N}}\subset H$. Then the following conditions are equivalent.

\begin{enumerate}
\item The sequence $F$ is an $AB$-frame.

\item $\mathcal{D}(C_{AB})=H$, $B^{1/2}C_{AB}\in\mathcal{B}(H,\ell^2)$, and $\mathcal{R}\bigl((B^{1/2}C_{AB})^*\bigr)=\mathcal{R}(A^{1/2})$.
\end{enumerate}
\end{theorem}

\begin{proof}
(1) Suppose that $F$ is an $AB$-frame with frame bounds $\alpha$ and $\beta$. By definition, $\mathcal{D}(C_{AB})=H$. Put $T=B^{1/2}C_{AB}$. For every $f\in H$, we have $\|Tf\|^2=\|B^{1/2}C_{AB}f\|^2=\|C_{AB}f\|_B^2$ and $\|A^{1/2}f\|^2=\|f\|_A^2$. Therefore, the upper frame inequality gives $\|Tf\|^2\leq\beta\|A^{1/2}f\|^2\leq\beta\|A\|\|f\|^2$. Hence $T=B^{1/2}C_{AB}\in\mathcal{B}(H,\ell^2)$.

The two frame inequalities can now be written as $\alpha\|A^{1/2}f\|^2\leq\|Tf\|^2\leq\beta\|A^{1/2}f\|^2$ for every $f\in H$. Since $\|Tf\|^2=\langle T^*Tf,f\rangle$ and $\|A^{1/2}f\|^2=\langle Af,f\rangle$, it follows that $\alpha A\leq T^*T\leq\beta A$.

From $\alpha A\leq T^*T$, we obtain $A\leq\alpha^{-1}T^*T$. By Douglas' range inclusion theorem, $\mathcal{R}(A^{1/2})\subseteq\mathcal{R}(T^*)$. Similarly, from $T^*T\leq\beta A$, Douglas' theorem gives $\mathcal{R}(T^*)\subseteq\mathcal{R}(A^{1/2})$. Consequently, $\mathcal{R}(T^*)=\mathcal{R}(A^{1/2})$. Since $T=B^{1/2}C_{AB}$, we conclude that $\mathcal{R}\bigl((B^{1/2}C_{AB})^*\bigr)=\mathcal{R}(A^{1/2})$. Thus, condition $(2)$ holds.

(2) Conversely, suppose that $\mathcal{D}(C_{AB})=H$, $T=B^{1/2}C_{AB}\in\mathcal{B}(H,\ell^2)$, and $\mathcal{R}(T^*)=\mathcal{R}(A^{1/2})$. The range equality gives both $\mathcal{R}(A^{1/2})\subseteq\mathcal{R}(T^*)$ and $\mathcal{R}(T^*)\subseteq\mathcal{R}(A^{1/2})$.

By Douglas' range inclusion theorem, the inclusion $\mathcal{R}(A^{1/2})\subseteq\mathcal{R}(T^*)$ implies that there exists $\lambda>0$ such that $A\leq\lambda T^*T$. Hence $\lambda^{-1}A\leq T^*T$. Similarly, the inclusion $\mathcal{R}(T^*)\subseteq\mathcal{R}(A^{1/2})$ implies that there exists $\mu>0$ such that $T^*T\leq\mu A$. Therefore, $\lambda^{-1}A\leq T^*T\leq\mu A$.

It follows that, for every $f\in H$, $\lambda^{-1}\langle Af,f\rangle\leq\langle T^*Tf,f\rangle\leq\mu\langle Af,f\rangle$. Since $\langle Af,f\rangle=\|f\|_A^2$ and $\langle T^*Tf,f\rangle=\|Tf\|^2=\|C_{AB}f\|_B^2$, we obtain $\lambda^{-1}\|f\|_A^2\leq\|C_{AB}f\|_B^2\leq\mu\|f\|_A^2$ for every $f\in H$. Thus, $F$ is an $AB$-frame with frame bounds $\alpha=\lambda^{-1}$ and $\beta=\mu$.
\end{proof}
\printbibliography[
heading=bibintoc,
title={References}
]
%\addcontentsline{toc}{section}{References}
%\bibliographystyle{plainnat}
%\bibliography{ref}
\end{document}